\UseRawInputEncoding
\documentclass[12pt]{article}
\usepackage{amsthm,amssymb,amsmath}

\newcommand{\ep}{\varepsilon}
\newcommand{\A}{\mathbb A}
\newcommand{\B}{\mathbb B}
\newcommand{\C}{\mathbb C}
\newcommand{\D}{\mathbb D}

\newcommand{\K}{\mathbb K}

\newcommand{\Q}{\mathbb Q}

\newcommand{\R}{\mathbb R}

\newcommand{\Z}{\mathbb Z}

\newcommand{\sB}{\mathcal B}
\newcommand{\sD}{\mathcal D}

\newcommand{\sL}{\mathcal L}
\newcommand{\sM}{\mathcal M}

\newcommand{\sP}{\mathcal P}
\newcommand{\sU}{\mathcal U}

\newcommand{\sdiv}{{}_{\mathrm{div}}}
\newcommand{\smod}{{}_{\mathrm{mod}}}
\newcommand{\umod}{{}^{\mathrm{mod}}}
\newcommand{\umm}{{}^{\mathrm{mm}}}

\newcommand{\1}{^{-1}}

\newcommand{\pt}{\mathrm{pt.}}
\DeclareMathOperator{\bDiv}{\mathcal{D}iv}

\DeclareMathOperator{\cent}{center}

\DeclareMathOperator{\Div}{Div}

\DeclareMathOperator{\GLC}{GLC}

\DeclareMathOperator{\mult}{mult}

\DeclareMathOperator{\Supp}{Supp}

\newcommand{\eqdef}{\mathrel{\,\stackrel{\mathrm{def}}{=}\,}}

\theoremstyle{plain}
 \newtheorem{cor}{Corollary}
 \newtheorem*{gener_sdiv}{General properties of divisorial part of adjunction}
 
 \newtheorem{prop}{Proposition}
 \newtheorem{propdf}{Proposition-Definition}
 \newtheorem{thm}{Theorem}

\theoremstyle{definition}
 \newtheorem{conj}{Conjecture}
 \newtheorem{df}{Definition}

\theoremstyle{remark}
 \newtheorem{ex}{Example}

\title{Log adjunction: moduli part}

\author{V.V.~Shokurov
\thanks{Partially supported by
NSF (grant DMS-0701465).
2010 {\em Mathematical Subject Classification} 14E30.}}

\begin{document}

\maketitle

\begin{abstract}
Upper moduli part of adjunction is introduced and
its basic property are discussed.
The moduli part satisfies the BP
in the case of rational multiplicities and
is nef in the maximal case.
\end{abstract}

Usually we assume that the base field $k$ is algebraicaly closed of
characteristic $0$, e.g., $k=\C$.

We start from the divisorial part.

\begin{df} \label{df-sdiv}
There are two equivalent approaches to define the {\em divisorial part of adjunction\/}
or {\em discriminant\/}: due to Kawamata \cite{Kaw} and due to Ambro \cite{Am}.
The definition of the divisorial part of adjunction $D\sdiv$ on $Z$ supposes,
that $X/Z$ is surjective (in codimension $1$ over $Z$), and
the following {\em generic lc\/} property
\begin{description}
  \item[$\GLC$:]
the pair $(X/Z,D)$ is lc over the generic points of the base $Z$.

\end{description}
The definition of its b-divisorial version $\D\sdiv$ of $Z$ supposes
additionally that $X/Z$ is proper and
$(X,D)$ is a log pair \cite[7.1 and Remark~7.7]{PSh}.
Notice also that
$(\D\sdiv)_Z=D\sdiv$ and
$\D\sdiv=\sum d_W W$, where $d_W$ are defined by \cite[Construction~7.2 and Remark~7.7]{PSh}.
\end{df}

Recall the definition of the {\em divisorial pull-back\/}
$f^\circ$ \cite[2.4]{PSh}.
Let $f\colon X\to Z$  be a surjective morphism of
normal irreducible varieties (or algebraic spaces).
Then the homomorphism of Weil $\R$-divisors
$$
f^\circ\colon \Div_\R Z\to \Div_\R X
$$
is uniquely linearly extended from
the prime Weil divisors.
For a prime Weil divisor $D$ on $Z$, put
$f^\circ D\eqdef f^*D$
over the generic point of $D$
and $0$ outside.
Note that $D$ is Cartier near
its generic point.

The homomorphism $f^\circ$ has a natural and unique
extension on the b-$\R$-divisors:
$$
f^\circ\colon \bDiv_\R Z\to \bDiv_\R X,
$$
such that for every morphism $f'\colon X'\to Z'$,
birationally equivalent to $f$ over $Z$, and
for every b-$\R$-divisor $\sD\in \bDiv_\R Z$,
$f'^\circ(\sD_{Z'})=(f^\circ \sD)_{X'}$
holds in codimension $1$ over $Z'$.
In other words, the difference
$f'^\circ(\sD_{Z'})-(f^\circ \sD)_{X'}$
is truly exceptional over $Z'$ \cite[Definition~3.2]{Sh00}.
In general
$f'^\circ(\sD_{Z'})=(f^\circ \sD)_{X'}$
does not hold everywhere even for a Cartier b-divisor $\sD$,
e.g., on prime truly exceptional divisors
with respect to $f'$.

On the $\R$-Cartier b-$\R$-divisors the pull-back $f^\circ$
is exactly well-known $f^*$ \cite[7.2]{Sh00}:
$$
f^\circ\sD=f^*\sD.
$$
Indeed, by definition, for every birational
base change $Z'$,
$f'^\circ \sD=f'^*\sD$
in codimension $1$ over $Z'$.
Additionally, b-$\R$-divisors $f'^\circ\sD$ and $f'^*\sD$ are vertical,
and, for every vertical prime
b-divisor $D$ of $X$, there exists a morphism $f'\colon X'\to Z'$,
birationally equivalent to $f$ over $Z$ such that
$D$ and $f'D$ are prime divisors on $X'$ and $Z'$ respectively.

Recall also the following properties.

\begin{gener_sdiv}[cf.~{\cite[Lemma~7.4]{PSh}}] \label{gener_prop_div_sdiv}
Let $(X/Z,D)$ be a log pair as in Definition~\ref{df-sdiv} and
$W$ be a prime divisor on $Z$.
The $\R$-divisor $D\sdiv$ and b-$\R$-divisor $\D\sdiv$ satisfy
the following properties:
\begin{description}

\item[\rm (1)\/]
{\em birationality:\/} $d_W=\mult_W D\sdiv=\mult_W\D\sdiv$ is independent of
a crepant model of $(X,D)$ over $Z$ and $\D\sdiv$
is also independent of a model of $Z$;

\item[\rm (2)\/]
{\em semiadditivity:\/} for any $\R$-Cartier divisor
$\Delta$ on $Z$ the log pair $(X/Z,D')$ with $D'=D+f^*\Delta$
satisfies $\GLC$ and equalities
$D'\sdiv=D\sdiv+\Delta,\D'\sdiv=\D\sdiv+\overline{\Delta}$ hold;

\item[\rm (3)\/]
$(X,D)$ is lc (respectively klt) over the generic point of $W$ if and only if
$d_W\le 1$ (respectively $<1$);

\item[\rm (4)\/]
{\em effectiveness:\/} if $D$ is effective over the generic
point of $W$ then $d_W\ge 0$;
so, $D\sdiv\ge 0$ if $D\ge 0$;

\item[\rm (5)\/]
{\em rationality:\/} if $D$ is a  $\Q$-divisor then
$D\sdiv,\D\sdiv$ are respectively $\Q$-, b-$\Q$-divisors;

\item[\rm (6)\/]
{\em boundary:\/} if $D$ is an $\R$-(respectively $\Q$-)boundary then
$D\sdiv$ is an $\R$-(respectively $\Q$-)boundary;
a similar statement holds for subboundaries
(cf. (3) above).
\end{description}
\end{gener_sdiv}

Notice that in~(3) the relative klt property over the generic point
of $W$ means the klt property in prime b-divisors of $X$
with the center $W$.

\begin{proof}
Immediate by definition.
\end{proof}

\begin{cor} \label{lc_div_adj}
Let $(X/Z,D)$ be a log pair under
the adjunction assumption $\GLC$.
Suppose that BP holds for the divisorial part of adjunction
$\D\sdiv$.
Then the divisorial part of adjunction exactly preserves
lc singularities:
\begin{description}
\item[\rm (1)]
$(X,D)$ is lc if and only if
so does $(Z,D\sdiv)$;
\item[\rm (2)]
$(X,D)$ is klt over $Z$ if and only if
so does $(Z,D\sdiv)$.

\end{description}
We can omit BP.
Then lc and klt properties on $Z$ are determined
directly by the b-$\R$-divisor $\D\sdiv$.

\end{cor}

\begin{proof}
The last phrase of the statement means that
the pair $(Z,\D\sdiv)$ is lc, if $\mult_P\D\sdiv\le 1$
for every prime b-divisor $P$ of $Z$, that is,
$\D\sdiv$ is a b-subboundary.
Respectively, the klt property over $Z$
means: $\mult_P\D<1$ for every prime b-divisor $P$ of $X$,
vertical over $Z$, where
$\D=\B(X,D)$.
So, if $\D\sdiv$ satisfies BP, stable over $Z$, then
the lc property holds for $(Z,D_Z)$
with the trace $D_Z=(\D\sdiv)_Z$ and
the klt property holds for $(X,D)$ over $Z$.

The proof follows immediately from
General property~(3).

\end{proof}

\begin{df}
Let $(X/Z,D)$ be a log pair with proper surjective $X/Z$ and
under $\GLC$.
Then a {\em moduli part\/} $\sM^X$
is easily defined as a b-$\R$-divisor
$$
\sM\eqdef\sM^X\eqdef\sD\umod\eqdef\K+\D-f^\circ(\K_Z+\D\sdiv),
$$
where $\D$ is the codiscrepancy b-$\R$-divisor of $(X,D)$ and
$\K=\K_X,\K_Z$ are canonical b-divisors of $X$ and of $Z$ respectively.
As a canonical b-divisor $\K$ the moduli part $\sD\umod$ is defined
up to a linear equivalence on $X$.
\end{df}

So, the {\em log adjunction\/} holds:
$$
\K+\D=f^\circ(\K_Z+\D\sdiv)+\sD\umod.
$$

Recall, that  the {\em codiscrepancy\/} b-$\R$-divisor
$\D=\B(X,D)=-\A(X,D)$
in the sum $\K+\D$ corresponds to the log pair $(X,D)$
and is $\overline{K+D}-\K$, that is,
$\K+\D=\overline{K+D}$.
For a crepant pair  $(X',D')$ of $(X,D)$, $\D'=\D$.

Bold vs script denotes very special behaviour
of a b-codiscrepancy and (anti)similar for
a canonical b-divisor.
All other b-$\R$-divisors, including, b-$\R$-Cartier ones
will be denoted by calligraphic capital letters.

So, on every model of $X'$ of $X$ over $Z$
$$
M'\eqdef(\sM)_{X'}=D'\umod\eqdef(\sD'\umod)_{X'}=
K_{X'}+D'-f^\circ(\K_Z+\D\sdiv)_{X'},
$$
where $D'=\D_{X'}$  is the trace of $\D$ on $X'$ and
is a crepant divisor of $D$ on $X'$ if $(X',D')$ is
a crepant log pair of $(X,D)$.
Moreover, if a morphism $f'\colon X'\to Z'$ is
birationally equivalent to $f$ over $Z$, then
$$
M'=K_{X'}+D'-f'^\circ(K_Z+D'\sdiv)=
K_{X'}+D'-f'^*(K_{Z'}+D'\sdiv)
$$
in codimension $1$ over $Z'$.
This allows to calculate the moduli part $M'$
immediately in codimension $1$ over $Z'$.
In general, for truly exceptional divisors
with respect to $f'$,
none of pull-backs $f'^\circ,f'^*$
give an exact formula everywhere.
However, subsequent blowups of the variety $X'$ and
of the base $Z'$ allows to determine the moduli part.
That is the meaning of b-($\R$-)divisors.

\begin{prop} \label{bp_equiv_Car}
b-$\R$-Divisor $\D\sdiv$ satisfies {\em BP}, or,
equivalently, $\K_Z+\D\sdiv$ is an $\R$-Cartier b-$\R$-divisor
if and only if
the moduli part $\sM$ is an $\R$-Cartier b-$\R$-divisor.

More precisely,
let $f'\colon X'\to Z'$ be a proper morpism
birationally equivalent to $f$ over $Z$ such that
$\K_X+\D$ is stable over $X'$.
Then if $\K_Z+\D\sdiv$ is an $\R$-Cartier b-$\R$-divisor, stable over $Z'$,
then $\sM$ is an $\R$-Cartier b-$\R$-divisor,
stable over model $X'$, and
$\sM=\overline{M'}$, where
$$
M'=K_{X'}+D'-f'^*(K_{Z'}+D'\sdiv).
$$
Conversely,
if $\sM$ is an $\R$-Cartier b-$\R$-divisor, stable over
the model $X'$ and additionally
$f'$ has equidimensional fibers,
then
{\em BP} holds for $\D\sdiv$ with stability over $Z'$.
\end{prop}

In the proposition we have two kinds of stabilities.
For an $\R$-Cartier b-$\R$-divisor $\sD$ of $X$,
the {\em stability over\/} a model $X'$ of $X$ means
that $\sD_{X'}$ is $\R$-Cartier and  $\sD=\overline{\sD_{X'}}$.
Respectively, the {\em stability} of BP, for an b-$\R$-divisor $\D$
{\em over\/} $X'$, means that $(X',\D_{X'})$ is a log pair and
$\D=\B(X',\D_{X'})$.
Notice that either of stabilities over $X'$ implies
the same over every model of $X$ over $X'$.
Equivalently, each stability holds over every
sufficiently high model of $X$ if the stability holds on
over some of models.
So, by Hironaka if a proper morphisms $f'\colon X'\to Z'$
is birationally equivalent to $f$ over $Z$ (or $Z'$)
then replacing $X'$ by a higher model of $X$ we can suppose
that $(X',D')$ is a crepant model of $(X,D)$ and
satisfies BP with stability over $X'$.

\begin{proof}
According to the last remark, required $f'$ exists for every
model $Z'$ of $Z$.
For the converse statement, a flattering due to Hironaka
allows to construct $f'$ with equidimensional fibers
for some sufficiently high models $Z'$ over $Z$.

If $\K_Z+\D\sdiv$ is $\R$-Cartier, stable over $Z'$, then
immediately by definition and our assumptions
\begin{align*}
\sM=&\K+\D-f^\circ(\K_Z+\D\sdiv)=
\overline{K+D}-f^*(\overline{K_{Z'}+D'\sdiv})=\\
&\overline{K_{X'}+D'}-\overline{f'^*(K_{Z'}+D'\sdiv)}=
\overline{K_{X'}+D'-f'^*(K_{Z'}+D'\sdiv)}.
\end{align*}
Conversely, let $\sM$ be $\R$-Cartier, stable over $X'$, and
$f'$ be with equidimensional fibers.
Then the $\R$-Cartier property of the b-$\R$-divisor $\K+\D\sdiv$
and its stability over $Z'$
follows from the following two statements.

For a proper morphism $f\colon X\to Z$
of normal varieties, with equidimensional fibers, the equality of pull-backs
$f^\circ=f^*$ holds on $\R$-Cartier divisors, wherein
$D$ on $Z$ is $\R$-Cartier if and only if
$f^\circ D$ is $\R$-Cartier on $X$.
General hyperplane sections of $X$ reduce the proof
to the case of a finite morphism $f$.

Let $f\colon X\to Z$ be a
composition of surjective morphisms
of normal varieties $g\colon X\to Y$
and $h\colon Y\to Z$,
and $D,E$ be $\R$-divisors on $Y$ and $Z$ respectively such that
$E$ is $\R$-Cartier and $g^\circ D=f^*E$
in codimension $1$ over $Y$.
Then $D=h^*E$, in particular, $D$ is also $\R$-Cartier and
$g^*D=f^*E$.
Reduce to the case $E=0$ for $D:=D-h^*E$.
Then $D=0$ too.

\end{proof}

If there exists an $\R$-Cartier b-$\R$-divisor $\sL$ such that
$\K+\D\sim_\R f^*\sL$ then
a {\em descent\/} of the moduli part $\sM^X$ on the base $Z$
is well-defined.
It is a {\em moduli part on base\/}:
$$
\sM^Z\eqdef\sD\smod\eqdef\sL-\K_Z-\D\sdiv.
$$
So,
$$
f^\circ\sM^Z=f^*\sL-f^\circ(\K_Z+\D\sdiv)\sim_\R
\K+\D-f^\circ(\K_Z+\D\sdiv)=\sM^X,
$$
that gives {\em log adjunction\/}
$$
\K+\D\sim_\R f^*(\K_Z+\D\sdiv+\sD\smod).
$$
The last equivalence also known as
the canonical class formula.

Of course, the pull-back $f^*\sL$
in the definition of $\sM^Z$ can be replaced
on more general $f^\circ\sL$.
However, this does not give anything new.
Indeed, if $f^\circ\sL$ is an $\R$-Cartier b-$\R$-divisor then
$\sL$ is also an $\R$-Cartier b-$\R$-divisor and $f^\circ\sL=f^*\sL$.
Additionally, according to our assumptions,
the b-$\R$-divisor $\K+\D$ is $\R$-Cartier and
the $\R$-linear equivalence preserves the last property.
(The $\R$-Cartier property is  preserved also
for the numerical equivalence $\equiv$ over $Z$
that allows to define a numerical version of log adjunction.)

\begin{prop}
The moduli part $\sD\smod$ is an $\R$-Cartier b-$\R$-divisor
if and only if the divisorial part $\D\sdiv$ satisfies {\em BP\/}.
\end{prop}

\begin{proof}
By definition $\sD\smod$ is an $\R$-Cartier b-$\R$-divisor
if and only if $\K_Z+\D\sdiv$ is $\R$-Cartier.
The last property is equivalent to BP for
the b-$\R$-divisor $\D\sdiv$.
\end{proof}

Unlike $\sM^X$,  the moduli part $\sM^Z$ on the base is defined
up to  an $\R$-linear equivalence.
However, $\sM^Z$ is defined up to a $\Q$-linear equivalence if
$\K+\D\sim_\Q f^*\sL$ for an $\R$-Cartier b-$\R$-divisor $\sL$
(possibly, with nonrational multiplicities).
More precisely, $\sM^Z$ is defined up to an $n$-linear equivalence if
$\K+\D\sim_n f^*\sL$.
Respectively in the log adjunction,
the equivalence $\sim_\R$ can be replaced by
$\sim_\Q$ or, more precisely, by $\sim_n$:
$$
\K+\D\sim_n f^*(\K_Z+\D\sdiv+\sD\smod).
$$

In general, according to Examples~\ref{exs_bir_finite}, (5-6) and
Proposition~\ref{bp_equiv_Car} b-$\R$-divisors $\sD\umod,\sD\smod$
are not $\R$-Cartier.

\begin{ex} \label{exs_bir_finite}
(1) Let $(X/Z,D)$ be a log pair with a birational contraction $f\colon X\to Z$,
that is,
$f$ is a proper birational map of normal varieties.
Then $D\sdiv=f_*D$ and $D\umod=0$.
Moreover, $\sD\umod=0$ and is a b-divisor and
the divisorial part $\D\sdiv$ satisfies BP.
The naive moduli part $K+D-f^*(K_Z+D)$ is
the discrepancy $\R$-divisor of $(Z,D)$ on $X$, if
$(Z,D)$ is a log pair and $f_*D=D$ as b-$\R$-divisors, equivalently,
the $\R$-divisor $D$ does not have exceptional components and is
a proper birational preimage of $f_*D$ with respect to $f$.
This situation corresponds to the definition of discrepancies
for the pair $(Z,D)$ and the discrepancies are determined
for prime divisors on the blowup $X/Z$.
In general, however, more natural to present the codiscrepancy b-$\R$-divisor
$\D=\B(X,D)=-\A(X,D)$ as the divisorial part of adjunction
for $(X/Z,D)$:
$$
\D=\D\sdiv.
$$
In particular, BP for $\D\sdiv$ is stable over $Z$, that is, $\K_Z+\D\sdiv$
is an $\R$-Cartier b-$\R$-divisor, stable over $Z$, if
$K+D\equiv 0$ over $Z$ and $K_Z+f_*D$
is an $\R$-Cartier $\R$-divisor on $Z$, equivalently,
$(X/Z,D)$ is a $0$-pair or
$(X,D)$ and $(Z,f_*D)$ are crepant over $Z$.
In this situation $\D\sdiv=\B(Z,f_*D)=-\A(Z,f_*D)$.

Notice that the birational map $f$ should be proper to define
a moduli part of adjunction.
Indeed, a birational map $f$ is surjective
(at least in codimension $1$) for any base change
if and only if the map is proper.

(2)
Let $(C/C',D)$ be a log pair with a surjective map
$f\colon C\to C'$ of normal curves,
separable, if the characteristic of
the base field  $k$ is positive.
Then for every point $p\in C'$ the geometric fiber $f^*p$
can be identified with the divisor of the fiber $f^*p=\sum m_i p_i$,
where the sum runs over all point $p_i$ of the fiber $f\1p$ and
$m_i=\mult_{p_i}f$ denotes the multiplicity of $f$ in $p_i$.
Denote respectively by $r_i$ the ramification index of
$f$ in $p_i$.
(If the characteristic  of $k$ does not divide $m_i$
then $r_i=m_i-1$.)
Let $d_i=\mult_{p_i}D$ be the multiplicity of
the $\R$-divisor $D$ in $p_i$.
Then the multiplicity $d'=\mult_pD\sdiv$
of the divisorial part of adjunction
can be determined by the formula
$$
d'=\max\{\frac{r_i+d_i}{m_i}\}.
$$
In characteristic $0$, if
$D$ has the same multiplicities $d_i=d\le 1$ over $p$,
the multiplicities of divisorial and moduli parts are
$$
d'=\frac{m-1+d}m
\text{ and }
\mult_{p_i}D\umod=(d-1)(1-\frac{m_i}m)
$$
respectively, where $m=\max\{m_i\}$.
Since the moduli part is defined only up to
an $\R$-linear equivalence,
the last formula is meaningful only for a complete curve.
In particular, if $D=0$ then $d'=(m-1)/m$.
�� �⮬ ����쭠� ����
$$
D\umod\le 0
$$
and is equal to $0$ if and only if
all $m_i=m$, for instance:
$C/C'$ is a Galois covering.
Next Example~(3) generalizes the last statement.

Notice that $D\umod=0$ also if all $d_i=1$
for $m\ge 2$ or all $d_i=d$ for $m=1$
(cf. the maximal property in
Proposition-Definition~\ref{maximal_log_pair}).

(3) Let $f\colon X\to Z$ be a (finite) Galois covering and
$(X/Z,D)$ be a log pair with an invariant $\R$-divisor $D$.
Then $\sD\umod=0$ by Example~(2).
Indeed, on one hand, hyperplane sections
reduce the determination of the moduli part in prime divisors
to a curve case.
On the other hand, we can blow up every given prime divisor
preserving the assumptions.

(4) Let $f:X\to Z$ be an unramified  double covering
of surfaces and $C$ be a nonsingular curve on the base $Z$ which
splits on the the covering $X$, that is, $f^*C=C_1+C_2$, where
$C_1,C_2$ are nonsingular curves on $X$.
First consider a log pair $(X/Z,D)$ with a divisor $D=C_1$.
Then by Example~(2) $D\sdiv=C$, the b-divisor $\D\sdiv=\B(Z,C)$
satisfies BP and stable over $Z$.

Remove now a closed point $p\in C_1$ on $X$.
Then for the log pair $(X'/Z,D)$ with $X'=X\setminus p$,
the map $X'\to Z$ is surjective and surjective for
every birational base change.
So, the b-divisor $\D\sdiv$ is well-defined but
does not satisfies BP.
More precisely,
$\D\sdiv=\B(Z,0)$ over $f(p)$ and $=\B(Z,C)=\B(Z,0)+\overline{C}$ over
the other points of the base $Z$.
This is why for log adjunction we usually assume
that $f$ is proper.

(5) Let $f:X\to Z$ be an unramified double covering
of surfaces and $C_1,C_2$ be nonsingular curves on the base $Z$,
which intersect transversally in a single point $p\in Z$ and
split on $X$, that is, $f^*C_1=D_1+D_1'$ and
$f^*C_2=D_2+D_2'$, where
$D_1,D_1',D_2,D_2'$ are nonsingular curves on $X$.
We suppose also that $D_1\cap D_2=D_1'\cap D_2'=\emptyset$.
Take a log pair $(X/Z,D)$ with an $\R$-divisor $D=d_1 D_1+d_2 D_2+
D_1'+D_2'$, where $d_1,d_2$ are real numbers $<1$ and
linearly independent over the rational numbers, that is,
the equality $a_1 d_1+a_2 d_2=a$ � $a_1,a_2,a\in\Q$, is possible
only for $a_1=a_2=a=0$.
Then the b-$\R$-divisor $\D\sdiv$ does not satisfy BP.

More precisely, as in Example~(4) above
the b-$\R$-divisor $\D\sdiv$ does not satisfy BP
over $p\in Z$.
Indeed, for the blowup $Z'\to Z$ of
the nonsingular point $p$,
the multiplicity of $\D\sdiv$ in the exceptional divisor $E$ is
$d=\max\{d_1,d_2\}$ and $d<1$ (but $D\sdiv=C_1+C_2$).
Denote by $X'=X\otimes_Z Z'$ the corresponding blowup of $X$.
On the covering $X'$ of $Z'$, $E$ splits into two
exceptional curves of the 1st kind $E_1,E_2$, where $E_1$
intersects the proper transform of $D_1$.
Then by semiadditivity (2) in General properties \cite[Lemma~7.4]{PSh}
the pair $(X/Z,D)$ can be replaced by the log pair $(X'/Z',D')$ with
the $\R$-divisor $D'=D+E_1+(1+d_2-d_1)E_2$, where
$D$ denotes the proper birational preimage
of the $\R$-divisor $D$ on $X'$.
We suppose also that $d=d_1> d_2$.
The new divisorial part of adjunction as the old one
satisfies BP everywhere except for the point $p'=C_1\cap E$,
where $C_1$ denotes the proper birational preimage of $C_1$ on $Z'$.
The new divisorial part of adjunction stabilizes on same blowups
if such exist.
The multiplicities $d_1,1+d_2-d_1$ of the new pair are
linearly independent over the rational numbers.
The prove concludes the induction on the number of blowups
required to get a stability of BP.
The process will newer stops.
This contradicts BP.

The same argument shows BP for
every $\Q$-divisor $D$ (cf. Theorem~\ref{Q_div_adj} below).

Similarly, one can construct
an example with a b-$\Q$-divisor
$\D\not=\B(X,D)$ instead of $D$ without BP,
the divisorial part of adjunction of which $\D\sdiv$
satisfies BP.

(6) Let $\sM_5$ denote the moduli space of
stable rational curves with $5$ marked points,
$\sU_5\to \sM_5$ be the corresponding universally family and
$\sP_1,\dots,\sP_5$ be sections corresponding
to the marked points.
The family is a smooth three dimensional (nonstandard)
conic bundle over a nonsingular surface $\sM_5$.
Let $D_1,D_1'$ be divisors on $\sU_5$ sweeping
respectively by components $C,C'$ of stable curves
$C\cup C'$ with points $p_1,p_2,p_3\in C,p_4,p_5\in C'$, where
$p_i=\sP_i\cap(C\cup C')$.
Similarly, divisors $D_2',D_2$ on $\sU_5$ are sweeping
respectively by components $C,C'$ of stable curves
$C\cup C'$ with points $p_1,p_2\in C',p_3,p_4,p_5\in C$.
As in Example~(5) take a log pair $(\sU_5/\sM_5,D)$
with an  $\R$-divisor $D=d_1 D_1+d_2 D_2+
D_1'+D_2'$, where $d_1,d_2$ are real numbers $<1$ and
linearly independent over the rational numbers.
Then the divisorial part of adjunction $\D\sdiv$
does not satisfy BP over the (closed) point $p$ of the base $\sM_5$,
corresponding to the stable curve $C_1\cup C_2\cup C_3$ with
points $p_1,p_2\in C_1,p_3\in C_2,p_4,p_5\in C_3$.
The b-$\R$-divisor $\D\sdiv$ has the same multiplicities over $p$
as the corresponding b-$\R$-divisor in Example~(5).
Adding to $D$ sections $\sP_i$ with
arbitrary multiplicities $\le 1$,
one can suppose that $D$ is $\R$-ample over $\sM_5$
with the same b-$\R$-divisor $\D\sdiv$.

Similarly one can construct an example with a $\R$-divisor $D$ such that
$K_{\sU_5}+D\sim_\R 0$ over $\sM_5$ and
$\D\sdiv$ does not satisfy BP.
(According to \cite[Theorem~8.1]{PSh}
the horizontal part of $D$ should be not effective, cf. also Example~\ref{max_mod_0_pair}
below.)
It is sufficient to find such multiplicities $a_1,\dots,a_5\le 1$
that $D:=D+\sum a_i \sP_i\equiv 0$ on
every irreducible component of the curve over the point $p$.
This gives an example over a neighborhood of $p\in\sM_5$.
Adding vertical prime divisors with appropriate multiplicities one
can find an example over $\sM_5$.
More precisely, one can find easily required multiplicities $a_i$,
when the multiplcities $d_1,d_2$ are close to $1$.
In this case one can pick up multiplicities $a_i$
close to $1/2$ for $i\not=3$ and  to $0$ for $i=3$
(actually $a_3<0$).

Again for rational $D$ BP holds.

\end{ex}

\begin{thm} \label{Q_div_adj}
Let $(X/Z,D)$ be a log pair with a $\Q$-divisor $D$,
proper $X/Z$ and under GLC of Definition~\ref{df-sdiv}.
Then the divisorial part of adjunction $\D\sdiv$ is
well-defined and always satisfies BP.
\end{thm}

A proof below uses the following fact.

\begin{prop}[Transitivity of divisorial part of adjunction]
\label{tras_div_adj}
Let
$$
f\colon X\stackrel{h}{\twoheadrightarrow}
Y\stackrel{g}{\twoheadrightarrow} Z
$$
a composition of two proper surjective morphisms and
$(X/Z,D)$ be a log pair under the adjunction assumption
$\GLC$ over $Z$.
Then the pair $(X/Y,D)$ also satisfies $\GLC$.
Suppose that the divisorial part of adjunction $\D_Y$
of $(X/Y,D)$ satisfies BP, stable over the base $Y$.
Then the {\em transitivity\/} holds:
$(Y/Z,D_Y)$ is also a log pair, satisfying the adjunction assumption
$\GLC$, where $D_Y=(\D_Y)_Y$ the trace of
the b-divisor $\D_Y$ on $Y$, and
$$
\D_Z=(D_Y)_Z
$$
holds, where $\D_Z,(D_Y)_Z$ denote respectively
divisorial parts of adjunction
of pairs $(X/Z,D),(Y/Z,D_Y)$.

The usage of the base $Y$ with stability of BP can be omitted.
Then, in the transitivity formula, the divisor $D_Y$ should be replaced
by the trace $D_{Y'}=(\D_Y)_{Y'}$ on a model $Y'$ of $Y$ over $Z$,
over which the stability of BP holds.
\end{prop}

\begin{proof}
If $(X,D)$ is lc over an open subset
$U\subset Z$ then the same holds over
the open set $g\1 U\subset Y$.
By surjectivity of $g$ the preimage $g\1 U$
is not empty if $U$ is not empty.

The transitivity follows from more precise result --
General property~(3).
It is sufficient to verify the transitivity in each prime b-divisor
$W$ of $Z$.
By General properties~(1-2)
we can suppose that $W$ is a prime divisor on $Z$ and
$d_W=\mult_W\D_Z=1$.
Then by General property~(3), $(X,D)$ is lc but not klt
over the generic point of $W$.
After blowing up of $X$ and changing of $D$ on its crepant $\R$-divisor,
we can suppose that $\mult_V D=1$ for some prime divisor $V$ on $X$ over $W$,
that is, $f(V)=W$.
We can suppose also that $h(V)$ is a prime divisor on $Y$.
Again by General property~(3), $d_{f(V)}=\mult_{f(V)} D_Y=1$ and
$(Y,D_Y)$ is not klt over the generic point of $W=g(h(V))$.
On the other hand, $(Y,D_Y)$ is lc over the generic point of $W$
by Corollary~\ref{lc_div_adj}
(and the open property of lc).
Thus $\mult_W{(D_Y)_Z}=1$ too.

\end{proof}

In the proof of Theorem~\ref{Q_div_adj}
we use also the following construction.

\begin{ex}[ Crepant pull-back] \label{crepant_pull_back}
Let $(Z,D_Z)$ be a log pair and $f\colon X\to Z$ be a morphism of normal
varieties, separable, finite and surjective over the generic point
(separable alteration).
Then there exits a natural and unique divisor $D$ on $X$,
converting the variety $X$ into a log pair $(X,D)$ {\em crepant\/}
to $(Z,D_Z)$, that is, for every canonical divizor
$K_Z=(\omega_Z)$ on $Z$, where $\omega_Z$ is a nonzero rational
differential form of the top degree on $Z$,
$$
K+D=f^*(K_Z+D_Z),
$$
holds, where $K=(f^*\omega)$ is a canonical divisor on $X$.
In this case by Example~\ref{exs_bir_finite}, (2) $D\sdiv=D_Z$
and $D\umod=D\smod=0$.
If the map $f$ is proper and surjective then
$\D\sdiv=\B(Z,D_Z)$ and $\sD\smod=0$.
In other words, $\D\sdiv$ satisfies
BP and stable over $Z$.

\end{ex}

\begin{cor} \label{addit_mod_adj}
$$
\sM_{X/Z}\sim \sM_{X/Y}+f^\circ\sM_{Y/Z}
$$
(actually, $=$ for appropriate canonical b-divisors of $X,Y,Z$),
where $\sM_{X/Z},\sM_{X/Y},\sM_{Y,Z}$ are respectively
the moduli part of adjunction for $(X/Z,D),(X/Y,D), (Y/Z,\D_Y)$.
\end{cor}

\begin{proof}
\begin{align*}
\sM_{X/Z}=&\K+\D-f^\circ(g^\circ(\K_Z+\D_Z))=
\K+\D-f^\circ(\K_Y+\D_Y-\sM_{Y/Z})=\\
&\K+\D-f^\circ(\K_Y+\D_Y)+f^\circ\sM_{Y/Z}=\sM_{X/Y}+f^\circ\sM_{Y/Z}.
\end{align*}
Since the moduli parts are defined up to a linear equivalence we can
choose required canonical b-divisors to have $=$ instead of $\sim$.
\end{proof}

\begin{proof}[Proof of Theorem~\ref{Q_div_adj}]
$\D\sdiv$ is well-defined by \cite[7.1-2]{PSh}
(cf. also Definition~\ref{df-sdiv} above).
Using the birational nature of $\D\sdiv$, namely,
General properties (1) and (5),
and by Proposition~\ref{tras_div_adj} one can reduce
the proof to two cases:
morphism $f: X\to Z$ is finite or
is a family of curves.

First we verify BP for the finite morphisms.
In this part of the proof, the assumption,
that the base field  $k$ has the characteristic $0$,
is essentially important.
For simplicity one can suppose that $k=\C$ is
the field of complex numbers.
According to Hironaka after
an appropriate blowup of the base and
replacing $X$ by an induced normal base change
one can assume that $Z$ is nonsingular with reduced
($1$ are the only nonzero multiplicities) divisor $\Delta$
with only simple normal crossing and such that
$f$ is only ramified and $D$ is supported over $\Supp\Delta$.
General hyperplane sections and dimensional induction
reduce the problem to a verification of BP
over a neighborhood of a finite set of isolated closed points $z\in Z$.
Over a sufficiently small neighborhood of every such point $z$
in the classical topology,
the covering $f$ is a union of simple branches.
Again by the transitivity of Proposition~\ref{tras_div_adj}
it is sufficient to verify BP in the simple maping case
when the fiber $f\1 z=X_z$ consists of a single point,
and in the case of unramified map $f$.
In the first case, up to an analytic isomorphism,
the pair $(X/Z,D)$ is toric, that is, $X\to Z$ is
a toric finite map with an invariant divisor $D$ on $X$.
The invariant divisor on the base $Z$ is $\Delta$.
By General properties (2-3)
one can suppose that the divisorial part of adjunction
is reduced: $D\sdiv=\Delta$.
Then $D$ is also reduced and BP holds over $Z$.
(Actually, in this toric situation,
the rationality of $D$ does not matter.)

The second case, with an unramified map,
is a bit harder.
Again by the transitivity of Proposition~\ref{tras_div_adj}
we consider the case of two sheets: $X=X_1\cup X_2, X_1\cap X_2=\emptyset$  and
the map $X\to Z$ consists of two isomorphisms
$X_1,X_2\to Z$.
General properties (2-3)
allow to suppose that $D\sdiv=\Delta$.
However, this time we know only that the multiplicities of the divisor $D$
are not exceeding $1$ over $\Supp\Delta$ and are equal to $0$ otherwise.
If all multiplicities of $D$ over $\Supp \Delta$
are equal to $1$, BP holds.
Otherwise there exists a multiplicity $d_i<1$
over a prime divisor of $\Supp\Delta$.
By construction and our assumptions every such multiplicity is rational and
they form a finite set $\{d_i\}$.
Hence there exists a positive integer $m$ such that
all numbers $m d_i$ are integral.
The following algorithm allows to find a model of $Z$
over which $\D\sdiv$ is stable and BP holds.

Put $d=\min\{d_i\}$.
Let $D_i$ be a prime divisor on $X$
over $\Supp\Delta$ with the multiplicity $\mult_{D_i}D= d_i=d$.
By construction and our assumptions $d<1$.
If $D_i$ intersects on $X$ all prime divisors $D_j$
over $\Supp\Delta$ with multiplicities $d_j<1$ and
with the image $f(D_j)$ intersecting $f(D_i)$ on $Z$ then
BP holds over a neighborhood of  ��� $f(D_i)$
(to verify this one can use Example~\ref{exs_bir_finite}, (2)).
(Here as usually a prime divisor means closed irreducible subvariety
of codimension $1$, but not only its generic point.)
Moreover, the multiplicity $d_i$ can be  replaced by $1$,
not changing $\D\sdiv$.
After finitely many steps either there are no prime divisors $D_i$
over $\Supp\Delta$ with the multiplicity $d$, or there is another prime divisor
$D_j$ over $\Supp\Delta$ with multiplicity $d_j<1$, with
$D_i\cap D_j=\emptyset$, but
$f(D_i)\cap f(D_j)\not=\emptyset$.
The last intersection is a nonsingular subvariety of the base $Z$ of codimension $2$.
In the first case take new $d=\min\{d_i\}$.
The algorithm terminates because $d< 1$ and
$m d\in \Z$.
In the second case blow up the intersection $f(D_i)\cap f(D_j)$ and
respectively the covering $X$.
Let $X'\to Z'$ be the induced unramified  covering
with blown up divisors $E_1,E_2,E$ on $X_1',X_2',Z'$ respectively,
where $X_1',X_2',Z'$ are blowups of varieties  $X_1,X_2,Z$.
To be more precise suppose that $D_i\subset X_1$.
Let $D,D_i,D_j$ denote proper birational transforms
of those divisors on $X'$ and $D'$ denote the divisor on $X'$,
corresponding to the crepant transform of the pair $(X,D)$.
Then $D_i\subset X_1',D_j\subset X_2'$ and
$$
\mult_{E_1} D'=d\le \mult_{E_2}D'=d_j<1.
$$
Again by Example~\ref{exs_bir_finite}, (2)
$\mult_{E}D'\sdiv=\max\{d,d_j\}=d_j$.
Now we replace the pair $(X/Z,D)$ by $(X'/Z',D'+(1+d-d_j)E_1+E_2)$
and return to the beginning of the algorithm.
The validity of BP this does not change.
(The new divisor $\Delta'=\Delta+E$.)
We contend that the algorithm stops when $\{d_i\}=\emptyset$ that
gives BP with stability over the constructed base.
Indeed, if $d_j=d$ then a new prime divisor over $\Supp\Delta$ with $d_i<1$
will not appear and one intersection $f(D_i)\cap f(D_j)\not=\emptyset$
disappear.
If $d_j>d$ then exacly one new prime divisor $E_1$ over $\Supp\Delta$ will appear
with the multiplicity $1>1+d-d_j>d$.
Notice that again $m(1+d-d_j)\in \Z$.
Hence a new prime divisor $D_j$ with the multiplicity $d$ will not appear and
after finitely many steps $D_i$ intersects all
prime divisors $D_j$ over $\Supp\Delta$ with multiplicities $d_j<1$ for which
$f(D_i)\cap f(D_j)\not=\emptyset$.
In this case, as we did it before we increase
the multiplicity $d_i=d$ up to $1$.

In conclusion consider the case with
the map $f$ being a family of curves.
After an appropriate surjective proper base change $Z'\to Z$
with $\dim Z'=\dim Z$ (alteration) one can suppose that the family
is maximally good.
More precisely, there exists a commutative diagram
$$
\begin{array}{ccc}
X'&\rightarrow & X\\
f'\downarrow & & f\downarrow\\
Z'&\rightarrow & Z
\end{array}
$$
such that $f'$ is a projective family of (semi)stable curves
with the smooth generic fiber and
the preimage of horizontal prime divisors of
$\Supp D$ is a union of sections $S_i$ of the family $f'$
\cite[Theorem~2.4]{dJ}.
(Sections $S_i$ are prime divisors too.)
Maps $X'\to X$ and $Z'\to Z$ are finite generically.
Let $D'$ be the crepant pull-back of Example~\ref{crepant_pull_back}.
By construction $D'$ is a $\Q$-divisor.
Then by the transitivity of Proposition~\ref{tras_div_adj}
for the composition,
$X'\to X\to Z$ gives the same divisorial part of adjunction
for the pair $(X'/Z,D')$ as for $(X,D)\to Z$, that is,
gives $\D\sdiv$.
On the other hand, the transitivity of Proposition~\ref{tras_div_adj}
for the composition
$X'\to Z'\to Z$ and already established BP for
the map $Z'\to Z$, BP for $\D\sdiv$
follows from BP for the divisorial part of adjunction
of the pair $(X'/Z',D')$.
We contend that BP holds for $(X'/Z',D')$
over $Z'$.
By \cite[Theorem~2.4, (vii)(b) and the property~2.2.2]{dJ},
$X'\to X$ is unramified generically over $Z$ or $Z'$.
The horizontal part of the divisor $D'$ has the same
multiplicities in the sections $S_i$ as in the corresponding
prime divisors of $X$ in $D$.
In our situation
the assumption $\GLC$ means that all $\mult_{S_i} D'\le 1$
for all sections $S_i$.
This follows from $\GLC$ for $(X/Z,D)$ by construction
and \cite[Remark~2.5]{dJ}.
Again by \cite[Remark~2.5 and Theorem~2.4]{dJ}, after
an additional blowup of $Z'$ and a normal base change of $f'$, we can suppose that $X'/Z'$
is toriodal with invariant $\Supp D'$.
Thus $D'\sdiv$ can be computed by an lc threshold with
a vertical divisorial lc center in $\Supp D'$.
More precisely,
to determine $D'\sdiv$ one can suppose that
the base $Z'$ is a curve.
By construction all sections $S_i$ and
vertical divisors, including vertical  prime components
of $\Supp D'$, form a divisor with toroidal crossings.
Thus in the determination one can reduce
all horizontal prime components
of $\Supp D'$ in $S_i$, that is,
to replace their multiplicities by $1$.
Adding locally (in classical or etale topology)
(malti)sections $S_i$ with $\mult_{S_i} D'=0$,
preserving mutually disjoint and toroidal property of (multi)sections $S_i$,
one can suppose that $S=\sum S_i$ is ample over $Z'$ and
is in the reduced part of $D'$.
(The ample divisor intersects any vertical divisor.)
By Example~(2) this reduce the determination of
$\D'\sdiv$ to that of
the divisorial adjunction of the log pair $(S/Z',D'_S)$,
where $S=\cup S_i\to Z'$ is a toroidal covering and
$D'_S$ is the adjunction of $D'$ on $S$ supported in the invariant divisor.
Under the base change the crepant divisor change of $D_S'$
will be the adjunction of the crepand change of $D'$ on a modification of $S'$.
Thus this case also follows from already known one
for finite (toroidal) morphisms\footnote{
F.~Ambro informed the author that he has an alternative
proof for toroidal morphisms with invariant $\Supp D'$.}.

\end{proof}

Theorem~\ref{Q_div_adj} and Examples~\ref{exs_bir_finite}, (5-6)
show that, in general BP is unstable with respect to
multiplicities.
However, BP holds as a certain good limit, that is, $\D\sdiv$
is always {\em pseudo\/}-BP.
Examples~\ref{exs_bir_finite}, (5-6) also shows
that conditions on vertical component multiplicities of  $D$
are important.
From the birational point of view over $Z$,
this should not be so important that shows the next result.
Its proof is also a preparation to the proof of our main construction
in Proposition-Definition~\ref{maximal_log_pair}.

\begin{thm} \label{torific}
Let $(X/Z,D)$ be a log pair under $\GLC$.
Then there exists a log pair $(X'/Z',D')$ such that
\begin{description}
\item[\rm (1)]
morphisms $X\to Z,X'\to Z'$  are birationally equivalent;
and
\item[\rm (2)]
the pair $(X'/Z',D')$ is crepant to $(X/Z,D)$ over the generic point
of $Z$;
\item[\rm (3)]
$(X'/Z',D')$ satisfies $\GLC$ and
\item[\rm (4)]
$\D\sdiv'$ satisfies BP.
\end{description}
Hence $\D\sdiv$ satisfies BP over
the generic point of $Z$.
\end{thm}

The last property is not surprising (cf. Proposition-Definition~\ref{maximal_log_pair}
below).

\begin{proof}
By Proposition~\ref{tras_div_adj}, Theorem~\ref{Q_div_adj} and Stein factorization
we can suppose that $X'\to Z'$ is a contraction and
to construct rational $\D\sdiv'$ or, moreover, integral.

We can replace $(X/Z,D)$ by a log pair $(X'/Z',D')$
which satisfies~(1-3); and
the crepant property~(3) holds everywhere.
Moreover, the morphisms $f'\colon X'\to Z'$ is toroidal
and the divisor $D'$ does not intersect toroidal embedding, that is,
is supported in the invariant divisor.
We can use for this a toroidalization of $X\to Z$
with the closed subset $\Supp D$ \cite[Theorem~2.1]{AK}.
Recall that by definition the toroidal embedding
is nonsigular on $X'$ and on $Z'$.
We contend that if we replace all vertical multiplicities
of divisor $D'$ outside of the toroidal embedding by $1$,
then BP~(4) will hold and is stable over $Z'$.
Notice that all horizontal prime components of $\Supp D'$ are
in the invariant divisor and have multiplicities $\le 1$ in $D'$.
Hence, in particular, $D\sdiv'=\Delta$ is
the complement to the toroidal embedding on $Z'$.
Notice that the last condition is preserved under
the modifications of the pair $(X'/Z',D')$ below.

Verify BP, stable over $Z'$, that is,
for every prime exceptional divisor $W$ of $Z'$,
the equality $d_W=\mult_W\D\sdiv'=b_W=\mult_W\B(Z',\Delta)$
holds.
We can establish this by induction on $1-b_W$.
Since $(Z',\Delta)$ is lc,
$b_W\le 1$ holds and by construction multiplicities $b_W$ are integers.
We will use also induction on dimension of
$\cent_{Z'}W$.
General hyperplane sections of the base allow to reduce
the dimension of such a center to $0$, that is,
to the case with a closed point $\cent_{Z'}W$.
Suppose first that $b_W=1$ and use induction
on the number of blowups, that is, we can suppose that
$W$ is the blowup of a closed point on $Z'$.
Then there exists a toroidal blowup of $X'$ with a center
over $W$.
Actually, we can toroidally  blow up both varieties $X',Z'$
simultaneously
(use subsequent toroidal resolution of $X'$ \cite[Proposition~4.4]{AK}).
As above $D\sdiv'=\Delta$ and $d_W=1$.

Suppose now that $b_W\le 0$.
If $\cent_{Z'}W$ is an lc center of $(Z',\Delta)$ then
we blow up it as above.
Thus we can suppose that $\cent_{Z'}W$ is not
an lc center.
Then we can add a general hyperplane
section $H$ through the center and preserve the toroidal property,
including the horizontal part of the divisor $D'$.
Replace $D'$ by $D'+f'^{*}H$.
This increases $b_W$ and completes induction
by General property~(2).

By the crepant assumption~(2) and General property~(1),
$\D\sdiv=\D\sdiv'$ holds over the generic point of the base.
This gives the last statement.
\end{proof}

In general the behaviour of multiplicities $d_W$
is unpredictable, except for general estimations
(see General properties~(3-6)).
However, in one important situation, a good behaviour of multiplicities of
divisorial part of adjunction is expected
(and already known in some cases).
See for details~\cite[Proposition~9.3, (i)]{PSh} and
\cite[6.8]{Sh20}.
Here we briefly recall only one main result about
hyperstandard multiplicities.

\begin{cor}
Let $\mathfrak R\subset[0,1]$ be a finite subset of rational numbers
and $d$ be a natural number.
Then there exists a finite subset of rational numbers $\mathfrak R'\subset [0,1]$
such that, for every $0$-pair $(X/Z,B)$ with $X/Z$ of weak Fano type,
$\dim X\le d$ and a boundary $B\in \Phi(\mathfrak R)$,
$$
B\sdiv\in\Phi(\mathfrak R')
$$
holds.
In particular, the divisorial part of adjunction $B\sdiv$
is also a boundary and
there exists a real number $\ep>0$ such that
every nonzero multiplicity of $B\sdiv$
is $\ge\ep$.
The number $\ep$ depends only on $\mathfrak R$ and $d$.
\end{cor}

The same follows for any proper $X/Z$ from
Index conjecture \cite[Conjecture~2]{Sh20}.

\begin{proof}
The last statement follows  from the dcc property of
$\Phi(\mathfrak R')$.
\end{proof}

With the effective b-semiampleness of the moduli part of adjunction $B\smod$
\cite[(7.13.3)]{PSh} \cite[Conjecture~3]{Sh20}
this reduces the birational effectiveness of an Iitaka map
\cite[Conjecture~1.1]{BZ}
to the same problem for a big log canonical divisor
(\cite[Theorem~1.3]{HMX} and cf. \cite[Theorem~1.3]{BZ}).
However, Index problem remains the main missing point
to complete the proof.

\begin{prop}
For a log pair $(X/Z,D)$ under $\GLC$,
the following properties are equivalent.
\begin{description}

\item[\rm (1)\/]
$(X/Z,D)$ is a {\em log stable\/} pair and $(Z,D\sdiv)$ is a log pair;

\item[\rm (2)\/]
$\D\sdiv$ satisfies BP stable over $Z$, in particular,
$(Z,D\sdiv)$ is a log pair;
and if

additionally $X\to Z$ has equidimensional fibers then
(1-2) are equivalent to

\item[\rm (3)\/]
$\sD\umod$ is $\R$-Cartier over $X$.

\end{description}
\end{prop}

The {\em log stable\/} property in (1) means that,
for every $\R$-Cartier divisor $\Delta$ on $Z$,
$(X,D+f^*\Delta)$ is lc if and only if
$(Z,D\sdiv+\Delta)$ is lc.
The global version of this property is equivalent
to the local one over a neighborhood of every point of $Z$.
Notice also that, by the log adjunction and the negativity lemma,
the b-nef over $Z$ property of $\sD\umod$
(cf. Conjecture~\ref{b-semiample} below)
implies that it is enough to assume the if part
of the log stability.

\begin{proof}
(1)$\Rightarrow$(2):
We need to verify that, for every prime exceptional divisor $W$ on $Z$,
$d_W=\mult_W\D\sdiv=b_W=\mult_W\B(Z,D\sdiv)$ holds.
Use a toroidalization $X'\to Z'$ as in the proof
of Theorem~\ref{torific} after a reduction to a contraction $X'\to Z'$.
Suppose additionally that the invariant part $\Delta\subset Z'$
contains $W$ and $g\1\Supp D\sdiv$ with all exceptional
divisors of a birational projective base change $g\colon Z'\to Z$.
Let $D_{Z'}=\B(Z,D\sdiv)_{Z'}$ be a codiscrepancy.
General property~ (2) allows to replace
$D'$ by $D'+f'^*F'$, where $F'$ is such an $\R$-divisor on $Z'$
supported on $\Delta$ that $\mult_W (D_{Z'}+F')=1$ and
$(Z',D_{Z'}+F')$ is lc.
After an extension of $\Delta$ we can suppose that
$F'\sim_\R 0/Z$ and $F'=g^*F$ for an $\R$-Cartier
divisor $F$ on $Z$.
The pair $(Z,D\sdiv+F)$ is also lc.
Thus by the log stability~(1) and construction
$(X,D+f^*F)$ and its crepant pair $(X',D'+f'^*F')$ are lc.
Now make a change $D:=D+f^*F$ and $F:=F':=0$.
Then $d_W\le b_W=1$
by General property~(3).
(In general $d_W>b_W$ is possible.)
Moreover, $d_W=b_W=1$.
Indeed, in our construction we can assume that
that $\B(Z,D\sdiv)$ has only one prime b-divisor $W$
over its center in $Z$ with $b_W=1$.
Thus if $d_W<1$ then $(X,D)$ is klt over the center.
Increasing singularity in the center we get
a contradicion with log stability.
In other words, the lc property can be replaced by
the relative klt property in the definition of the log stability.

(1)$\Leftarrow$(2) By General property~(2)
it is sufficient to verify the implication assuming BP
stabile over $Z$, that is,
$(X,D)$ is lc exacly when $(Z,D\sdiv)$ is lc.
If $(X,D)$ is lc then the b-divisor $\D\sdiv$ is lc,
that is, it is a subboundary, by General property~(3).
Hence $(Z,D\sdiv)$ is lc by BP stable over $Z$.
Conversely, if $(X,D)$ is not lc, then $\B(X,D)$
is not lc, that is, it is not a subboundary.
Hence the b-divisor $\D\sdiv$ also is not lc by General property~(3)
and because a prime b-divisor $W$ on $X$ with the multiplicity
$\mult_W\B(X,D)>1$ dominates a prime divisor on a blowup
of the base $Z$.
Thus $(Z,D\sdiv)$ is not lc by BP stable over $Z$.

(2)$\Leftrightarrow$(3) by Proposition~\ref{bp_equiv_Car}.
\end{proof}

Let $(X/Z,D)$ be a log pair the generic fiber of which
is a weakly lc pair.
Recall that {\em weakly lc\/} means in this situation that $(X_\eta,D_\eta)$ is lc
with a boundary $D_\eta$ and $K_{X_\eta}+D_\eta$ is nef
where $\eta$ is the general point of $Z$.
Then in the class of weakly lc pairs birationally equivalent
to $(X/Z,D)$ with respect to the base $Z$,
that is, isomorphic or crepant generically over $Z$,
the upper moduli part of adjunction has a largest value $\sD\umm$.
It is largest modulo $\R$-linear equivalence on $X$,
attained on log pairs naturally
related to moduli spaces of the generic fiber,
and will be constructed in Proposition-Definition~\ref{maximal_log_pair} below
modulo LMMP.
This allows to define a birationally invariant with respect to
the base $Z$ moduli part  of adjunction $\sD\umm=\sD\umm(D)=\sD\umm(X/Z,D)$
which satisfies certain remarkable properties
(e.g., Conjecture~\ref{b-semiample}).
However, as usually in mathematics we prefer the adjective
{\em maximal\/} instead of {\em largest\/}.

\begin{propdf}[Maximal log pair] \label{maximal_log_pair}
Let $(X/Z,D)$ be a weakly lc pair with a boundary $D$
over the generic point of $Z$.
Suppose that LMMP holds indimensions $\le \dim X$.
Then there exists a maximal log pair $(X_m/Z_m,B_m)$, a weakly lc pair
birationally equivalent to $(X/Z,D)$ with respect to the base $Z$.
The {\em maximal\/} property means the inequality
$\sB_m\umod\ge\sB'\umod$ modulo linear equivalence
for every weakly lc pair $(X'/Z',B')$,
birationally equivalent to $(X/Z,D)$ with respect to $Z$.

Moreover, $\sD\umm=\sB_m\umod$ is
a b-nef $\R$-Cartier b-$\R$-divisor.
\end{propdf}

For appropriate canonical divisors on $X$ and $Z$,
$\ge$ holds literally without the linear equivalence.

\begin{ex} \label{max_mod_0_pair}
Let $(X/Z,D)$ be a pair under $\GLC$ and
generically be a $0$-pair over $Z$.
Then it is a maximal log pair exactly when $(X/Z,D)$
is a $0$-pair everywhere over $Z$ with
a boundary $D$.
The last condition can be omitted
keeping the same maximal moduli part of adjunction
(cf. \cite[Proposition~13, (1)]{Sh20}).
In this situation
$\sD\umm=f^*\sD\smod$.
Thus $\sD\umm$ is a generalization of $f^*\sD\smod$
with consequent conjectures (cf. Conjecture~\ref{b-semiample} below).

\end{ex}

Moduli construction of $(X_m/Z_m,B_m)$ and a proof of
the proposition-definition actually need a weaker assumption on LMMP:
in $\dim X/Z+1$, and  will be explained in \cite{Sh13}.
The moduli approach does not use also the b-nef property of $\sD\umm$
(cf. Conjecture~\ref{b-semiample} below).

\begin{proof}[Proof of Proposition-Definition~\ref{maximal_log_pair}]
The nef property of $M_X=\sM\umm_X$ was established in \cite[Theorem~1.1]{ACShS}
using the theory of foliations for contractions $X/Z$.
The b-nef property of $\sM\umm$ follows from stability of $\sM\umm$ below.
Thus, for contractions $X/Z$, we can suppose the b-nef property.
Actually, we need the b-nef property in $\dim X-1$.
An alternative approach to the b-nef property will be discussed after
Conjecture~\ref{b-semiample} below.

Fix the birational class of $(X/Z,D)$, that is,
the class of pair $(X'/Z',B')$ which are birationally equivalent to
$(X/Z,D)$ with respect to the base $Z$.
More precisely, we consider the subclass of pairs $(X'/Z',B')$
which are weakly lc with a boundary $B'$ and
$(X'/Z',B')$ is crepant to $(X/Z,D)$ generically over $Z$ or $Z'$.

{\em Construction of $(X_m/Z_m,B_m)$.\/}
Suppose first that $X/Z$ is a contraction.
By \cite[Theorem~2.1]{AK} we can suppose that $(X/Z,B)$ is
toroidal with a nonsingular projective base $Z$ and
with a boundary $B$ such that
\begin{description}

  \item[]
$\Supp B$ is in the invariant divisor;

  \item[]
with the same multiplicities as $D$ generically over $Z$,
that is, in all those prime divisors nonexceptional on original $X$;
and

  \item[]
with multiplicities $1$ in all other invariant prime divisors.

\end{description}
By our assumptions we can construct a weakly lc pair $(X_m/Z_m,B_m)$
over $Z=Z_m$.
This is a required maximal model.
By construction $(X_m/Z_m,B_m)$ belongs to the considered class
weakly lc models.
Moreover, $\B_{m,}\sdiv{}_{,Z_m}=\Delta_m$ is the invariant divisor on $Z_m$.

In general, we do not know the existence of toroidilization (why not).
So, we use the following construction.
Take a Stein factorization
$$
X\stackrel{f}{\to} Y\stackrel{g}{\to} Z
$$
of $X/Z$ where $f,g$ are respectively a contraction and a finite morphism.
For an appropriate model of $X/Z$, we can suppose that $Z$ is nonsingular,
$g$ is toroidal and there is a morphism $\varphi\colon Y\to Y_m$ for
a maximal log pair $(X_m/Y_m,B_m)$ constructed above such that
the divisorial part $(\B_{m,}\sdiv)_Y$ is supported in the invariant
divisor $D$ of the toroidal finite morphism $g$.
We suppose also that $(X/Y,B)$ is isomorphic to $(X_m/Y_m,B_m)$ over $Y\setminus D$
isomorphic to $Y_m\setminus\varphi(D)$.
The crepant model $(X/Y,(\B_m)_X)$ of $(X_m/Y_m,B_m)$ is a weakly lc pair if
the trace $(\B_m)_X$ is effective.
Otherwise, we replace $(\B_m)_X$ by $(\B_m)_X+f^*E$, where
$E=D-(\B_{m,}\sdiv)_Y$.
By construction $E$ is effective, $(\B_{m,}\sdiv)_Y+E$ is the invariant divisor $D$.
By General property (2),
$D$ is the divisorial part of adjunction for $(X/Y,(\B_m)_X+f^*E)$.
By construction, the moduli part of adjunction for $(X/Y,(\B_m)_X+f^*E)$
is the same $\sD\umm$ and maximal that will be established below
(already known for contractions).
Then we can reconstruct $(X/Y,(\B_m)_X+f^*E)$ into
a weakly lc pair using LMMP over $Y$: take log resolution and
replace multiplicities of all exceptional divisors over $D$ by $1$ and
by $0$ otherwise, and apply LMMP over $Y$.
The maximal property below for contractions warrants that the constructed model is
crepant to $(X/Y,(\B_m)_X+f^*E)$ and is maximal over $Y$.

{\em Stability of $\sD\umm$.\/}
Since in the last construction the divisorial part of adjunction
for $(X/Y,(\B_m)_X+f^*E)$ is an integral invariant divisor $D$ then
the stability for $X/Z$ is equivalent to the stability for $X/Y$ by
Theorem~\ref{tras_div_adj} and Proposition~\ref{bp_equiv_Car}.
So, we can suppose that $X/Z$ is a contraction.
By definition $\sD\umm=\sB_m\umod$.
So, by Proposition~\ref{bp_equiv_Car}, it is enough to verify that
$\B_{m,}\sdiv$ satisfies BP stable over $Z_m$.
Using General properties, dimensional induction and
induction on the number of monoidal transformations
we can consider only one such transformation in the following situation.
Let $P$ be an invariant prime cycle and
$Z\to Z_m$ be a monoidal transformation in $P$.
The transformation is toroidal and the invariant
divisor $\Delta$ on $Z$ is the birational transform of $\Delta_m$
plus the exceptional divisor $E$ (over $P$) of the transformation.
To establish required stability it is enough to verify that
$\mult_E\B_{m,}\sdiv=1$, that is, $\Delta=\B_{m,}\sdiv{}_{,Z}$.
Taking hyperplane sections we can suppose that $P$ is
a closed point.
We can suppose also that $\Delta_m$ is sufficiently large and
$(X_m,B_m)$ itself (not only over $Z_m$) is a weakly lc model.

To compute the last divisorial part of adjunction we can take
a toroidal resolution of $X_m$ over $Z$ and
then construct a model $(X/Z,B)$ as in above Construction
using LMMP.
By construction $\Delta=\B\sdiv{}_{,Z}$.
Thus to verify that $\B\sdiv{}_{,m}=\B_{m,}\sdiv{}_{,Z}$ it
is enough (and actually necessary) to verify that
$(X,B)$ is a crepant model of $(X_m,B_m)$.
(As usually we consider such models over $\pt$)
Indeed, $(Z,\Delta)\to (Z_m,\Delta_m)$ is crepant.

On the other hand, the weakly lc pair $(X_m,B_m)$ can be constructed
from $(X,B)$ by LMMP and a crepant transformation.
If $\Delta$ is sufficiently large the only possible
curves negative with respect to $K+B$ are curves $C$ over $E$.
These curves are on reduced prime divisors of invariant part $V$.
Moreover, by above Construction and adjunction
$(C.K+B)=(C.K_V+B_V)=(C.\sB_V\umm)$ and $\ge 0$
by the b-nef property of $\sB_V\umm$ for $(V/E,B_V)$
by \cite[Theorem~1.1]{ACShS} and we need this property in dimension $\le \dim X-1$.
Hence there are no negative curves and $(X,B)$ is crepant to
$(X_m,B_m)$.

{\em The maximal property of $\sD\umm$.\/}
In particular, $\sD\umm$ modulo linear equivalence is
independent of construction.
Moreover, $\sD\umm$ is unique for fixed canonical divisors on $X$ and $Z$.
This can be verified directly using General properties.
Notice for this especially General property~(2) which
implies that $\sD'\umm=\sD\umm$ for $D'=D+f^*\Delta$.

If $X\to Z$ is a contraction, adding effective divisors we can
suppose that $B'\sdiv=B_{m,}\sdiv=\B_{m,}\sdiv{}_{,Z'}$.
If additionally $\B_{m,}'\sdiv$ is BP stable over $Z'$
then the required inequality follows from the property
that, for a larger divisor, its positive part in the Zariski
decomposition is larger too.
Perhaps, the same holds even if $\B\sdiv'$ does not satisfies BP.

In general, we use the above construction of $(X/Y/Z,B_m)$
with $B_m:=(\B_m)_X+f^*E)$,
where ${\B_{m,}}\sdiv=\overline{D}$ is BP stable over $Z$.
The maximal moduli part $\sM\umm$ over $Z$ is the same as over $Y$:
$$
\sD\umm=\K+\B_m-(f\circ g)^\circ(\K_Z+\overline{D_Z})=
\K+\B_m-f^\circ(g^\circ(\K_Z+\overline{D_Z}))=\K+\B_m-f^\circ(\K_Y+\overline{D}),
$$
where $D_Z$ is the invariant divisor for the toroidal morphism $g$.
Notice also that any weakly lc model over $Z'$ is a weakly lc model over $Y'/Z'$
with a birational morphism $Y\to Y'$ for sufficiently high model $Y$ because
any curve over $Y'$ is a curve over $Z'$.

We do not need to suppose that the divisorial part of adjunction $\B'\sdiv$ of $(X'/Y',B')$
is BP stable over $Y$.
The required maximal property follows from Corollary~\ref{addit_mod_adj}.
Indeed, the corollary
$$
\sB'_{X'/Z'}=\sB'_{X'/Y'}+\sM_{Y/Z},
$$
where $\sB'_{X'/Z'},\sB'_{X'/Y'},\sM_{Y,Z}$
denotes respectively the moduli part of adjunction for
$(X'/Z',B'),(X'/Y',B'),(Y/Z,\B'\sdiv)$.
However, in nonstable situation we consider a moduli part and
divisorial part of adjunction as b-divisors.
Since $\sB'_{X'/Y'}\le \sD\umm$, it is enough to verify that
$\sM_{Y/Z}\le 0$ for appropriate canonical b-divisors of $Y,Z$
(cf. Example~\ref{exs_bir_finite}, (2)).

By definition, construction and the lc property of $\B'\sdiv$ we can suppose that
$\B'\sdiv\le\overline{D}$ and $=\overline{0}$ over $Y\setminus D$
(cf. Theorem~\ref{torific}).
By General property (2) after increasing we can suppose that
the divisorial part of adjunction $\D'$ for $(Y/Z,\B'\sdiv)$ is
$\overline{D_Z}$.
Since $Y/Z$ is toroidal
$$
\sM_{Y/Z}=\K_Y+\B'\sdiv-g^\circ(\K_Z+\overline{D_Z})=
\K_Y+\B'\sdiv-\K_Y-\overline{D}=
\B'\sdiv-\overline{D}\le 0
$$
holds.
This concludes the proof of maximal property.

In other words, $\sD\umm$ is the positive b-divisor of
a relative lc b-divisor.

\end{proof}

\begin{conj} \label{b-semiample}
For every positive $a\in \R$,
$\sD\umm+a\sP$ is b-semiample and
effectively b-semiample if the moduli type of
irreducible components of generic fiber is bounded where
$\sP$ is a divisor of canonical polarization for
moduli of irreducible components of generic fiber.
Moreover, the Iitaka dimension of $\sD\umm$ is
at least the Kodaira dimension of the general fiber of $(X/Z,D)$
plus the variation of $(X/Z,D)$.
\end{conj}

The effective b-semiample property of $\sD\umm+a\sP$ is
meaningful even if $D$ and $a$ are not rational.
This can be explained in terms of geography of log models
(see Conjecture~\ref{loggeog} below and
\cite[Bounded affine span and index of divisor
in Section~12]{Sh20}).

A geometrical interpretation of the conjecture:
the corresponding contraction
gives a family of log canonical models of components of
fibers.
So, in addition to a rational morphism from $Y$ in
the Stein factorization $X\to Y\to Z$ to the
coarse moduli of irreducible components of fibers,
we have a rational morphism of $X$ to the family with log canonical models
in fibers whereas $\sP$ is ample on the base, $\sD\umm$ is
nef on the family and ample on its fibers.
In particular, the coarse moduli can be extended to
a fibration with corresponding log canonical models.

Notice that the conjecture implies the Kodaira additivity
in the strong form (with variation).
The semiampleness of $\sD\umm$ does not hold in general
according to \cite[Section~3]{K99}.

The conjecture implies also that $\sD\umm$ is b-nef.
This property naturally related to
the Viehweg positivity for the direct image of a relative dualizing sheaf and
the polarization $\sP$.
We already used the b-nef property in the proof of Proposition-Definition~\ref{maximal_log_pair}.
The author thinks that this property follows
from its special klt case with a $1$-dimensional base
by \cite[Thorem~0.1]{Am03} and adjunction.

Finally, we expect the following strong form of stability of $\sD\umm$
with respect to horizontal part of $D$ which is supposed to be
a boundary.

\begin{conj}[Geography of log adjunction] \label{loggeog}
Let $X/Z$ be a proper morphism, $\eta$ be the general point of
$Z$ and $S_i$ be distinct prime b-divisors of $X$.
Consider a polyhedron $\mathfrak P$ in the wlc geography of $\mathfrak N_{S_\eta}$
of $(X_\eta/\eta,\sum S_{i,\eta})$ \cite[Section~3]{ShCh}.
Then $\sD\umm(D)$ is linear on $\overline{\mathfrak P}$:
for every two (b-)divisors $D_1,D_2$ such that
$D_{1,\eta},D_{2,\eta}\in\overline{\mathfrak P}$ and any two
real numbers $w_1,w_2\in[0,1]$ with $w_1+w_2=1$,
$$
\sD\umm(w_1D_1+w_2D_2)=w_1\sD\umm(D_1)+w_2\sD\umm(D_2).
$$
In other words $\sD\umm$ is a piecewise linear function
of the horizontal part $D_\eta$ of $D$.
\end{conj}

Since $\mathfrak P$ is rational it is enough to verify Conjecture~\ref{b-semiample}
for $\Q$-divisors $D$.

\bigskip
\noindent
Department of Mathematics, the Johns Hopkins University,
Baltimore, Maryland 21218, the USA

\noindent
119991 Steklov Mathematical Institute RAS,
8 Gubkina st., 119991, Moscow, Russia

\noindent
e-mail: shokurov@math.jhu.edu

\end{document}